\documentclass[11pt, reqno]{amsart}
\usepackage{mathrsfs}
\usepackage{graphics}
\usepackage{latexsym,amssymb,amsmath,amscd,amscd,amsthm,amsxtra}
\usepackage{graphicx}
\usepackage{graphics,color}
\usepackage{amsfonts}
\usepackage[colorlinks,linkcolor=red]{hyperref}

   \def\lineno.sty{\texttt{\itshape lineno.sty}}
\def\Box{\vcenter{\vbox{\hrule\hbox{\vrule
     \vbox to 8.8pt{\hbox to 10pt{}\vfill}\vrule}\hrule}}}

\newcommand{\Ff}{{\mathbb F}}

\def\Tr{\operatorname{Tr}}

\newtheorem{thm}{Theorem}[section]
\newtheorem{lem}[thm]{Lemma}
\newtheorem{cor}[thm]{Corollary}

\newtheorem{prop}[thm]{Proposition}

\newtheorem{defn}[thm]{Definition}
\newtheorem{example}[thm]{Example}
\newtheorem{remark}{Remark}
\def\Tr{\operatorname{Tr}}

\numberwithin{equation}{section}

\begin{document}
%\begin{linenumbers}
%\pagewiselinenumbers
%\renewcommand\linenumberfont{\normalfont\small\bfseries}
%\setlength\linenumbersep{1cm}
 %\maketitle
 \thispagestyle{empty}
%\linenumbers
\title{Fractional revival on abelian Cayley graphs}

\author[X. Cao, G. Luo]{Xiwang Cao, Gaojun Luo}

\address{Xiwang Cao is with the School of Mathematical Sciences, Nanjing University of Aeronautics and Astronautics, Nanjing 210016, China, and he is also Key Laboratory of Mathematical Modeling and High Performance Computing of Air Vehicles(NUAA), MIIT, Nanjing 210016, China, email: {\tt xwcao@nuaa.edu.cn}}
\address{Gaojun Luo is with School of Physical and Mathematical Sciences, Nanyang Technological University, 21 Nanyang Link, Singapore 637371, Singapore, email: {\tt
gaojun.luo@ntu.edu.sg}}
%\address{Ying-Ying Tan is with School of Mathematics and Physics, Anhui Jianzhu University, Hefei 230601, China, and she is also with School of Mathematical Sciences, Anhui University, Hefei 230601, China, email: {\tt tansusan1@ahjzu.edu.cn}}
%\address{Li-Ping Wang is with Institute of Information Engineering, Chinese Academy of Sciences, Beijing 100093, Beijing, China, email: {\tt wangliping@iie.ac.cn}}
%\address{Qiang Wang is with School of Mathematics and Statistics, Carleton University, 1125 Colonel By Drive, Ottawa, Ontario K1S 5B6, Canada. email: {\tt
%wang@math.carleton.ca}}
\thanks{%Keqin Feng has supported by National Natural Science Foundation of China under grant numbers 11471178, 11571107 and the Tsinghua National Information Science and Technology Lab.,
Xiwang Cao has supported by National Natural Science Foundation of China under grant numbers 11771007, 12171241, %Ying-Ying Tan has supported by National Natural Science Foundation of China under grant numbers 11601003 and Natural Science Foundation of Anhui Province under grant number 1808085MA17.
}

\begin{abstract}
Fractional revival, known as a quantum transport phenomenon, is essential for entanglement generation in quantum spin networks. The concept of fractional revival is a generalization of perfect state transfer and periodicity on graphs. In this paper, we propose a sufficient and necessary condition for abelian Cayley graphs having fractional revival between any two distinct vertices. With this characterization, two general constructions of abelian Cayley graphs having fractional revival is presented. Meanwhile, we establish several new families of abelian Cayley graphs admitting fractional revival.

\end{abstract}


\maketitle
%\tableofcontents

\section{Introduction }

In quantum information, continuous-time quantum walk is an essential ingredient in transport problems of quantum spin networks. The study of perfect state transfer in quantum spin chains dates back to \cite{Bose2003}. Further research has been established in \cite{Albanese2004,Christandl2004,Christandl2005}. In \cite{ALASTAIR2010}, Kay provided a comprehensive survey on this topic. With entanglement generation in quantum spin chains, the transport problem is said to be fractional revival in \cite{Genest2016}. Investigating fractional revival is both of theoretical and practical significance as it generalizes two quantum transport phenomena, namely, perfect state transfer and periodicity.

Let $\Gamma=(V,E)$ be a connected simple graph with set $V$ of $n$ vertices and set $E$ of edges. Let $A=A(\Gamma)=(a_{u,v})_{u,v\in V}$ be the adjacency matrix of $\Gamma$ defined by
\begin{equation*}
  a_{u,v}=\left\{\begin{array}{cc}
                   1, & \mbox{ if $(u,v)\in E$}, \\
                   0, & \mbox{ otherwise}.
                 \end{array}
  \right.
\end{equation*}
The matrix $A$ is an $n\times n$ symmetric $\{0,1\}$-matrix and all eigenvalues of $A$ are real numbers. The transfer matrix of $\Gamma$ is defined by the following $n\times n$ unitary matrix
%Let $\Gamma=\Gamma(G;\alpha)$ be a Cayley color graph with adjacency matrix $A$. The transfer matrix of $\Gamma$ is defined by the following $n\times n$ matrix:
\begin{equation*}
  H(t)=\exp(\imath tA)=\sum_{s=0}^{+\infty}\frac{(\imath tA)^s}{s!}=(H_{u,v}(t))_{u,v\in V},\ t\in \mathbb{R}, \imath =\sqrt{-1},
\end{equation*}
where $\mathbb{R}$ is the set of real numbers.
\begin{defn}For two
distinct vertices $x$ and $y$ of $\Gamma$, if
\begin{equation}\label{f-1}
H(t)\mathbf{e}_x=\alpha \mathbf{e}_x+\beta \mathbf{e}_y
\end{equation}
for some complex numbers $\alpha$ and $\beta$ satisfying $\beta\neq 0$ and $|\alpha|^2+|\beta|^2=1$, then the graph $\Gamma$ admits
$(\alpha,\beta)$-fractional revival (or fractional revival, in short, FR) from $x$ to $y$ at time $t$, where $\mathbf{e}_x$ is  the standard vector whose $x$-th entry is $1$ and $0$ otherwhere. Factoring a common unimodular
phase factor, $(\alpha,\beta)$-fractional revival is turned into $e^{\imath \theta}(\alpha,\beta)$-fractional revival, where $\alpha$ and $\theta$ are real
scalars and $\beta$ is complex. If $\alpha=0$, then $(0,\beta)$-fractional revival is reduced to perfect state transfer (in short, PST). The graph $\Gamma$ is said to be periodic at some time $t$ if $\beta=0$ in \eqref{f-1}. \end{defn}

%It is known that if $\Gamma$ has PST from $u$ to $v$ at time $t$, $\Gamma$ also has PST from $v$ to $u$ at time $t$. Thus we say that as $\Gamma$ has PST between $u$ and $v$ at time $t$.

The basic problem on this topic is to find graphs having FR. More specifically, for a given connected simple graph $\Gamma=(V,E)$, and $u,v\in V$, define
\begin{eqnarray}\label{f-2}
  \nonumber T(u,v)&=&\{t\in \mathbb{R}: t>0, H(t)\mathbf{e}_u=\alpha \mathbf{e}_u+\beta \mathbf{e}_v\}.
\end{eqnarray}
 We want to determine $T(u,v)$ for any pair $\{u,v\}$ of vertices, or at least, to determine if $\Gamma$ has FR, that is, $T(u,v)$ is not empty for some pair $\{u,v\}$. PST on various graphs, which is a special case of FR, has been investigated by many researchers. The works in \cite{Basic2009,Cao2020,Cao2020a,Tan_2019,Luo2021,Cheung2011,Coutinho2016,Coutinho2015,Coutinho2015a,Godsil2012,Godsil2012a,Li2021,Zheng2018,Pal2016} and the references therein form a partial list of an already vast literature. With regard to PST, only a few attempts have been made to FR. Genest, Vinet and Zhedanov in \cite{Genest2016} systematically studied quantum spin chains with FR at two sites. Later in \cite{Bernard2018}, a graph exemplified the existence of FR between antipodes. Chan {\it et al.} in \cite{Chan2019} established a general approach to constructing graphs exhibiting FR between cospectral vertices. In addition, they characterized cycles and paths having FR. In a consecutive paper \cite{Chan_2020}, the same authors of \cite{Chan2019} discussed FR on graphs with adjacency matrices from the Bose-Mesner algebra of association schemes. A comprehensive study of FR on graphs between cospectral and strongly cospectral vertices was presented in \cite{Chan2020}. Unlike the work in \cite{Chan2020}, Godsil and Zhang in \cite{Godsil2021} proposed
 a class of graphs admitting FR between non-cospectral vertices. Given a fixed element $a$ of an abelian group $G$, Wang, Wang and Liu in \cite{Wang2022} gave a sufficient and necessary condition for Cayley graphs on $G$ having FR between vertices $x$ and $x+a$. The characterization in \cite{Wang2022} is incomplete as the element $a$ is fixed. Up to now, there is no sufficient and necessary condition for Cayley graphs on abelian groups admitting FR between any two distinct vertices.

 This paper aims at characterizing FR on Cayley graphs over abelian groups between any two distinct vertices. Our contribution begins by Theorem \ref{lem-2} in the next section with a sufficient and necessary condition for abelian Cayley graphs having FR between any two distinct vertices. Based on such a criteria, we propose two general constructions of abelian Cayley graphs admitting FR in Corollaries \ref{thm-1.61} and \ref{thm-2.81}. As a consequence, two families of abelian Cayley graphs admitting FR are provided in respective Propositions \ref{thm-1.6} and \ref{thm-2.8}. Furthermore, Propositions \ref{thm-1.6} and \ref{thm-2.8} tell us that the pool of feasible graphs having FR is larger than that of graphs admitting PST. In Section \ref{S3}, we propose the definition of plateaued functions over abelian groups. With the help of these functions, a family of abelian Cayley graphs having FR is built in Theorem \ref{thm-3.6}. In Section \ref{S4}, we discuss FR on cublike graphs. A flexible construction of cublike graphs admitting FR is given in Theorem \ref{thm-4.1}. Utilizing bent or semi-bent functions over finite fields of even characteristics, we present the last construction of cublike graphs having FR in Theorem \ref{bentconstruction}. Section \ref{S5} concludes this paper.

\section{A characterization of abelian Cayley graphs having FR}
Let $G$ be an abelian group of order $n$ and let $S$ be a subset of $G$ with $S=-S$ and $G=\langle S\rangle$. Let $\mathbb{C}$, $\mathbb{Z}$ and $\mathbb{Q}$ be, respectively, the set of complex numbers, the set of integers and the set of rational numbers. Let $\Gamma={\rm Cay}(G,S)$ be an abelian Cayley graph whose vertex set is $G$ and two vertices, $u,v$, form an edge of $\Gamma$ if $u-v\in S$. The adjacency matrix of $\Gamma$ is an $n\times n$ matrix $A=(a_{g,h})$ defined by
\begin{equation*}
  a_{g,h}=\left\{\begin{array}{cc}
                  1, & \mbox{ if $g-h\in S$}, \\
                  0, & \mbox{ otherwise}.
                \end{array}
  \right.
\end{equation*}
An automorphism of $\Gamma$ is a bijection from $G$ to $G$ preserving the adjacency. The set of automorphism of $\Gamma$ is denoted by ${\rm Aut}(\Gamma)$. For every $z\in G$, define $\sigma_z: G\rightarrow G; x\mapsto x+z$. It is easy to see that $\sigma_z\in {\rm Aut}(\Gamma)$ and $G$ is a subgroup of ${\rm Aut}(\Gamma)$.

Each finite abelian group $G$ can be decomposed as a direct sum of cyclic groups:
\begin{equation*}
  G=\mathbb{Z}_{n_1}\oplus \cdots\oplus \mathbb{Z}_{n_r},\ \ n_s\geq 2,
\end{equation*}
where $\mathbb{Z}_m=(\mathbb{Z}/m\mathbb{Z},+)$ is a cyclic group of order $m$. For every $x=(x_1,\cdots,x_r)\in G$ with $x_s\in \mathbb{Z}_{n_s}$, the mapping
\begin{equation*}
  \chi_x:G\rightarrow \mathbb{C},\  \chi_x(g)=\prod_{s=1}^r\omega_{n_s}^{x_sg_s}, \  g=(g_1,\cdots,g_r)\in G
\end{equation*}
is a character of $G$, where $\omega_{n_s}=\exp(2\pi i/n_s)$ is a primitive $n_s$-th root of unity in $\mathbb{C}$. The set of characters of $G$, denoted by $\hat{G}$, forms a group under the following operation:
\begin{equation*}
  \chi_g\chi_h: G\rightarrow \mathbb{C}; x\mapsto \chi_g(x)\chi_h(x), \forall g,h\in G.
\end{equation*}

Let $\hat{G}$ be the character group of $G$. Recall that for any abelian group $G$, $G$ is isomorphic to its dual group via $G\rightarrow \hat{G}, g\mapsto \chi_g$. It is known that the eigenvalues of $A$ are the real numbers
\begin{equation*}
  \lambda_g=\sum_{s\in S}\chi_g(s), \forall g \in {G}.
\end{equation*}

Let $P=\frac{1}{\sqrt{n}}(\chi_g(h))_{g,h\in G}$ and $E_z=p_z^*p_z$ with $p_z$ being the $z$-th column of $P$ and $p_z^*$ being the conjugate transpose of $p_z$. The spectral decomposition of $A$ is given by
\begin{equation}\label{f-3}
  A=\sum_{g\in G}\lambda_gE_g.
\end{equation}
Since $\{E_z: z\in G\}$ forms a set of idempotents, namely,
\begin{equation*}
  E_gE_h=\left\{\begin{array}{cc}
                  E_g, & \mbox{ if $g=h$},\\
                  0, & \mbox{otherwise},
                \end{array}
  \right.
\end{equation*}
we have the spectral decomposition of $H(t)$ as follows:
\begin{equation}\label{f-4}
  H(t)=\sum_{g\in G}e^{\imath \lambda_g t}E_g.
\end{equation}
Thus the $(g,h)$-th entry of $H(t)$ is
\begin{equation}\label{f-5}
  H(t)_{g,h}=\frac{1}{n}\sum_{z\in G}e^{\imath t\lambda_z}\chi_z(g-h).
\end{equation}
If FR happens on $\Gamma={\rm Cay}(G,S)$ from a vertex $x$ to a vertex $y$, then by (\ref{f-1}) and noticing that $H(t)$ is a unitary matrix, the $x$-th row of $H(t)$ is determined by
\begin{equation}\label{f-6}
  H(t)_{x,u}=\left\{\begin{array}{cc}
                      \alpha, & \mbox{ if $u=x$}, \\
                      \beta, & \mbox{ if $u=y$}, \\
                      0, & \mbox{otherwise}.
                    \end{array}
  \right.
\end{equation}
Since $A$ is an $G$-invariant matrix, namely, $a_{g+z,h+z}=a_{g,h}, \forall g,h,z\in G$, we know that $H(t)$ is also $G$-invariant.
Therefore, we have the following preliminary result.
\begin{lem}\label{lem-1}Let $G$ be an abelian group of order $n$ and let $S$ be a subset of $G$. Let $\Gamma={\rm Cay}(G,S)$ be a Cayley graph with the adjacency matrix $A$. Then $\Gamma$ has $(\alpha,\beta)$-FR from a vertex $x$ to a vertex $y$ at some time $t$ if and only if
\begin{equation}\label{f-9}
  H(t)=\alpha I_n+\beta Q,
\end{equation}
where $Q$ is a permutation matrix satisfying $Q^T=Q$ and $Q^2=I_n$, the identity matrix of order $n$. Moreover, $Q$ has no fixed points and $\alpha \overline{\beta}+\overline{\alpha}\beta=0$.
\end{lem}
\begin{proof}Suppose that $ H(t)=\alpha I_n+\beta Q$ for some $t$. By the fact that $Q$ is a permutation matrix with no fixed points, we can assume that there is some $a\neq 0$ such that
\begin{equation*}
  Q_{0,a}=1, Q_{0,b}=0, \forall b\neq 0.
\end{equation*}
Then we have
\begin{equation*}
  H(t)_{0,0}=\alpha, H(t)_{0,a}=\beta, \mbox{ and } H(t)_{0,b}=0, \forall b\neq 0, a.
\end{equation*}
Therefore the graph $\Gamma$ exhibits $(\alpha,\beta)$-FR from $0$ to $a$.

Conversely, if $\Gamma$ has $(\alpha,\beta)$-FR from $x$ to $y$, then
\begin{equation*}
  H(t)\mathbf{e}_x=\alpha \mathbf{e}_x+\beta \mathbf{e}_y,
\end{equation*}
which implies that
\begin{equation*}
  H(t)_{x,x}=\alpha, H(t)_{x,y}=\beta \mbox{ and } H(t)_{x,b}=0, \forall b\neq x, y.
\end{equation*}
Since $H(t)$ is $G$-invariant, we have that for every $z\in G$,
\begin{equation*}
  H(t)_{x+z,x+z}=\alpha, H(t)_{x+z,y+z}=\beta \mbox{ and } H(t)_{x+z,b+z}=0, \forall b\neq x, y.
\end{equation*}
When $z$ runs through $G$, we get that
 $H(t)=\alpha I_n+\beta Q$, where $Q$ is a permutation matrix with no fixed points. The matrix $Q$ is symmetric as $H(t)$ is a symmetric matrix. Since $H(t)$ is a unitary matrix, we obtain that
 \begin{equation*}
   H(t)\overline{H(t)}=I_n\Leftrightarrow \alpha \overline{\alpha}I_n+\beta \overline{\beta}Q^2+(\alpha \overline{\beta}+\overline{\alpha}\beta)Q=I_n.
 \end{equation*}
 Therefore, we have
 \begin{equation*}
   \left\{\begin{array}{l}
            Q^2=I_n, \\
           \alpha \overline{\beta}+\overline{\alpha}\beta=0.
          \end{array}
   \right.
 \end{equation*}
 This completes the proof.
\end{proof}

\textbf{From now on, we only consider the case of $(\alpha,\beta)$-FR with $\alpha\neq 0$. That is, we always assume that $\alpha \beta\neq 0$.}

Recall that $H(t)$ is $G$-invariant, if FR happens on $\Gamma={\rm Cay}(G,S)$ from a vertex $x$ to a vertex $y$, then we have
$H(t)_{x,u}=H(t)_{0,u-x}$, thus we can write
\begin{equation}\label{f-7}
  H(t)_{0,u}=\frac{1}{n}\sum_{z\in G}e^{\imath t\lambda_z}{\chi_z(u)}=\left\{\begin{array}{cc}
                      \alpha, & \mbox{ if $u=0$}, \\
                      \beta, & \mbox{ if $u=y-x$}, \\
                      0, & \mbox{otherwise}.
                    \end{array}
  \right.
\end{equation}
Denote $a=x-y$. According to \eqref{f-7} and its Fourier inversion, we get
\begin{equation}\label{f-81}
  e^{\imath \lambda_g t}=\alpha+\overline{\chi_a(g)}\beta, \forall g\in G.
\end{equation}
Since $\lambda_g$ is real for every $g\in G$, we derive that
\begin{equation*}
  1=e^{\imath \lambda_g t}\overline{e^{\imath \lambda_g t}}=(\alpha+\overline{\chi_a(g)}\beta)(\overline{\alpha}+{\chi_a(g)}\overline{\beta})=1+\overline{\chi_a(g)}\overline{\alpha}\beta+{\chi_a(g)}\alpha\overline{\beta},
\end{equation*}
which is equivalent to
\begin{equation}\label{f-10}
  \chi_a(g){\alpha}\overline{\beta}+\overline{\chi_a(g)}\overline{\alpha}{\beta}=0.
\end{equation}
Combing with $\overline{\alpha}\beta+\alpha\overline{\beta}=0$ in Lemma \ref{lem-1}, we have
\begin{equation}\label{f-11}
  \chi_a(g)=\overline{\chi_a(g)}, \mbox{ i.e., } \chi_a(g)=\pm 1, \forall g\in G.
\end{equation}
That is, $a$ is of order two in $G$. As a consequence, FR cannot happen on any abelian Cayley graph with the underlying group has an odd order.

Define two subsets of $G$ as follows,
\begin{equation}\label{f-12}
  G_0=\{g\in G: \chi_a(g)=1\}, G_1=\{g\in G: \chi_a(g)=-1\}.
\end{equation}
Then $G_0$ is a subgroup of $G$ of order $n/2$ and $G_1$ is a coset of $G_0$.
By (\ref{f-81}), we have
\begin{equation}\label{f-13}
  e^{\imath t \lambda_g}=\left\{\begin{array}{cc}
                                  \alpha+\beta, & \mbox{ if $g\in G_0$},\\
                                 \alpha-\beta, & \mbox{ if $g\in G_1$}.
                                \end{array}
  \right.
\end{equation}
Assume the polar decompositions of $\alpha$ and $\beta$ are respective $r_\alpha e^{\imath \theta_\alpha}$ and $r_\beta e^{\imath \theta_\beta}$. According to $|\alpha|^2+|\beta|^2=1$ and $\overline{\alpha}\beta+\alpha\overline{\beta}=0$, we have
\begin{equation}\label{f-14}
  \left\{\begin{array}{ll}
           r_\alpha^2+r_\beta^2 &=1 \\
           \theta_\alpha-\theta_\beta & =\frac{\pi}{2}+k\pi, \mbox{ for some $k\in \mathbb{Z}$}.         \end{array}
  \right.
\end{equation}
Thus
\begin{equation*}
  \alpha\pm\beta=e^{\imath \theta_\beta}((-1)^kr_\alpha \imath\pm r_\beta),
\end{equation*}
Write $r_\beta+(-1)^kr_\alpha \imath =e^{\imath \theta_0}$. Then $\theta_0\in \mathbb{R}$ and
\begin{equation}\label{f-15}
  \alpha+\beta=e^{\imath(\theta_0+\theta_\beta)}, \alpha-\beta=e^{\imath({\pi}-\theta_0+\theta_\beta)}.
\end{equation}
Write $t=2\pi T$ and $\theta_0=2\pi u_0, \theta_\beta=2\pi u_\beta$. Due to (\ref{f-13}), it holds that
\begin{equation}\label{f-16}
  T\lambda_g-u_0-u_\beta\in \mathbb{Z}, \forall g\in G_0, \mbox{ and } T\lambda_g+u_0-u_\beta\in \frac{1}{2}+\mathbb{Z}, \forall g\in G_1.
\end{equation}
Therefore,
\begin{equation*}
  T\left(\sum_{g\in G_0}\lambda_g+\sum_{g\in G_1}\lambda_g\right)-nu_\beta\in \frac{n}{4}+\mathbb{Z}.
\end{equation*}
Since $\sum_{g\in G_0}\lambda_g+\sum_{g\in G_1}\lambda_g=\sum_{g\in G}\lambda_g=0$, we have $u_\beta\in \mathbb{Q}$.

Now, we arrive at
\begin{equation*}
  \sum_{g\in G_0}\lambda_g=\sum_{g\in G_0}\sum_{s\in S}\chi_g(s)=\sum_{s\in S}\sum_{g\in G_0}\chi_g(s)=\sum_{s\in S}\sum_{g\in G}\frac{1+\chi_a(g)}{2}\chi_g(s)=\left\{\begin{array}{cc}
                                                                                                                                                                         n/2, & \mbox{ if $a\in S$,} \\
                                                                                                                                                                         0, & \mbox{ otherwise}.
                                                                                                                                                                       \end{array}
  \right.
\end{equation*}
Thus, $\sum_{g\in G_0}\lambda_g\in \mathbb{Z}$ and so $\sum_{g\in G_1}\lambda_g\in \mathbb{Z}$ as well.

Let $d=|S|$. Thanks to (\ref{f-16}), we deduce that
\begin{equation*}
  T\left(n/2\cdot d+\sum_{g\in G_1}\lambda_g\right)-n/2 \cdot u_\beta\in n/4+\mathbb{Z},
\end{equation*}
Thus we obtain that $T\in \mathbb{Q}$.
By (\ref{f-16}) again, we have $Td-u_0-u_\beta\in \mathbb{Z}$ which further implies that $u_0\in \mathbb{Q}$, and finally, we get that $\lambda_g\in \mathbb{Q}$ for every $g\in G$. Since $\lambda_g$ is an algebraic integer, we know that $\lambda_g\in \mathbb{Z}$. In other words, $\Gamma$ is an integral graph.
Meanwhile, since $T$ is rational, $e^{\imath t}$ is a root of unity.
Therefore, we have the following result.

\begin{thm}\label{lem-2}Let $G$ be an abelian group of an even order $n$ and let $S$ be a symmetric subset of $G$. Let $\Gamma={\rm Cay}(G,S)$ be a Cayley graph with the adjacency matrix $A$. Suppose that $\alpha$ and $\beta$ are two nonzero complex numbers. Then $\Gamma$ has $(\alpha,\beta)$-FR from a vertex $x$ to a vertex $y$ at some time $t$ if and only if the following conditions hold.

(1) $a=x-y$ is of order two;

%(2) $e^{\imath t}$ is a root of unity;

(2) $\Gamma$ is an integral graph;

(3) for every $g\in G_0$, $e^{\imath t \lambda_g}=\alpha+\beta$; for every $g\in G_1$, $e^{\imath t \lambda_g}=\alpha-\beta$,
where $G_0$ and $G_1$ are defined by (\ref{f-12}). \end{thm}
\begin{proof}The necessity has been proved by the foregoing discussion. Next, we prove the sufficiency. Since for every $g\in G_0$, $e^{\imath t \lambda_g}=\alpha+\beta$, and for every $g\in G_1$, $e^{\imath t \lambda_g}=\alpha-\beta$, by (\ref{f-5}), we get that, for every $x\in G$,
\begin{equation*}
  H(t)_{x,x}=\frac{1}{n}\sum_{z\in G}e^{\imath t \lambda_z}=\frac{1}{n}\left[\sum_{z\in G_0}(\alpha+\beta)+\sum_{z\in G_1}(\alpha-\beta)\right]=\alpha.
\end{equation*}
Meanwhile,
\begin{eqnarray*}
   H(t)_{x,x+a}&=&\frac{1}{n}\sum_{z\in G}e^{\imath t \lambda_z}\chi_z(a)\\
   &=&\frac{1}{n}\left[\sum_{z\in G_0}(\alpha+\beta)\chi_z(a)+\sum_{z\in G_1}(\alpha-\beta)\chi_z(a)\right]\\
   &=&\frac{\alpha}{n}\sum_{z\in G}\chi_z(a)+\frac{\beta}{n}\left[\sum_{z\in G_0}\chi_z(a)-\sum_{z\in G_1}\chi_z(a)\right].
\end{eqnarray*}
Due to $a$ is of order two, and $\sum_{z\in G}\chi_z(a)=0$, we get
\begin{equation*}
   H(t)_{x,x+a}=\frac{2\beta}{n}\sum_{z\in G_0}\chi_z(a)=\frac{2\beta}{n}\sum_{z\in G_0}1=\beta.
\end{equation*}
Moreover, for every $b\neq 0,a$, we have
\begin{eqnarray*}
   H(t)_{x,x+b}&=&\frac{1}{n}\sum_{z\in G}e^{\imath t \lambda_z}\chi_z(b)\\
   &=&\frac{1}{n}\left[\sum_{z\in G_0}(\alpha+\beta)\chi_z(b)+\sum_{z\in G_1}(\alpha-\beta)\chi_z(b)\right]\\
   &=&\frac{\alpha}{n}\sum_{z\in G}\chi_z(b)+\frac{\beta}{n}\left[\sum_{z\in G_0}\chi_z(b)-\sum_{z\in G_1}\chi_z(b)\right]\\
   &=&\frac{2\beta}{n}\sum_{z\in G_0}\chi_z(b)\\
   &=&\frac{2\beta}{n}\sum_{z\in G}\frac{1+\chi_z(a)}{2}\chi_z(b)\\
   &=&\frac{\beta}{n}\sum_{z\in G}\chi_z(a+b)\\
   &=&0.
\end{eqnarray*}
Therefore, we have $|\alpha|^2+|\beta|^2=1$, and $H(t)=\alpha I_n+\beta Q$, where $Q=\sum_{g\in G}e_ge_{g+a}^T$ is a symmetric permutation matrix. By Lemma \ref{lem-1}, $\Gamma$ has $(\alpha,\beta)$-FR at time $t$.
\end{proof}

\begin{remark}
According to Definition \ref{f-1}, the existence of $(\alpha,\beta)$-FR is equivalent to the existence of $e^{\imath \theta}(\alpha,\beta)$-FR with $\theta,\alpha$ being real numbers and $\beta$ being a complex number. Based on Theorem \ref{lem-2}, we claim that $\theta$ is a rational number and $\beta$ is a pure complex number if $\Gamma={\rm Cay}(G,S)$ has $e^{\imath \theta}(\alpha,\beta)$-FR. Due to Lemma \ref{lem-1}, it is easy to check that $\beta$ is a pure complex number. By \eqref{f-13}, we obtain that
\begin{equation*}
  e^{\imath t \lambda_g-\imath \theta}=\left\{\begin{array}{cc}
                                  \alpha+\beta, & \mbox{ if $g\in G_0$},\\
                                 \alpha-\beta, & \mbox{ if $g\in G_1$}.
                                \end{array}
  \right.
\end{equation*}
Since $\alpha$ is a real number and $\beta$ is a pure complex number, we deduce that $e^{\imath t (|S|+\lambda_g)-2\imath \theta}=1$ with $g\in G_1$. Thanks to Items (2) and (3) in Theorem \ref{lem-2}, $\theta$ is a rational number.
\end{remark}

In what follows, we continue to investigate the conditions in Theorem \ref{lem-2}.

Taking an fixed element $g_1$ in $G_1$, by (\ref{f-13}), we have
\begin{equation}\label{f-18}
  T(d-\lambda_g)\in \mathbb{Z} \mbox{ for every $g\in G_0$ and }T(\lambda_g-\lambda_{g_1})\in \mathbb{Z} \mbox{ for every $g\in G_1$}.
\end{equation}
Since $\Gamma$ is an integral graph, we define the following integers:
\begin{equation}\label{f-19}
  M_0=\gcd(d-\lambda_g: g\in G_0), M_1=\gcd(\lambda_{g_1}-\lambda_g: g\in G_1), M=\gcd(M_0,M_1).
\end{equation}
(\ref{f-18}) is equivalent to
\begin{equation*}
  TM_0, TM_1\in \mathbb{Z},
\end{equation*}
and thus
\begin{equation}\label{f-20}
  T\in \frac{1}{M}\mathbb{Z}.
\end{equation}
Next, we show that $M$ is a divisor of $n=|G|$.
\begin{lem}\label{lem-M} Let $G$ be an abelian group of order $n$ and let ${\rm Cay}(G,S)$ be an integral graph. Then $M$ is a divisor of $n$, where $M$ is defined by (\ref{f-19}).
\end{lem}
\begin{proof}Write
\begin{equation*}
  d=\lambda_g+M t_g \mbox{ for every $g\in G_0$ and $\lambda_g=\lambda_{g_1}+Ms_g$ for every $g\in G_1$, where $t_g,s_g\in \mathbb{Z}$}.
\end{equation*}
Take an element $z\in S$ with $z\neq 0,a$. Then
\begin{equation}\label{f-21}
  \sum_{g\in G}\lambda_g\overline{\chi_g(z)}=\sum_{s\in S}\sum_{g\in G}\chi_g(s-z)=n.
\end{equation}
Meanwhile,
\begin{eqnarray}\label{f-22}
 \nonumber \sum_{g\in G}\lambda_g\overline{\chi_g(z)}&=&\sum_{g\in G_0}(d-t_gM)\overline{\chi_g(z)}+\sum_{g\in G_0}(\lambda_{g_1}-s_gM)\overline{\chi_g(z)}\\
 &=&(d+\lambda_{g_1})\sum_{g\in G_0}\overline{\chi_g(z)}-M\sum_{g\in G_1}(t_g+s_g)\overline{\chi_g(z)}.
\end{eqnarray}
Moreover,
\begin{equation}\label{f-23}
  \sum_{g\in G_0}\overline{\chi_g(z)}=\sum_{g\in G}\frac{1+\chi_g(a)}{2}\overline{\chi_g(z)}=\sum_{g\in G}\chi_g(a-z)=0.
\end{equation}
(\ref{f-21})-(\ref{f-23}) imply that
\begin{equation}\label{f-24}
 n/M=-\sum_{g\in G_1}(t_g+s_g)\overline{\chi_g(z)}.
\end{equation}
The LHS of (\ref{f-24}) is in $\mathbb{Q}$, while the RHS is an algebraic integer. Thus $n/M\in \mathbb{Z}$.
\end{proof}

By Lemma \ref{lem-M}, we have the following result.

\begin{cor}
Let $G$ be an abelian group of order $n$ and let ${\rm Cay}(G,S)$ be an integral graph. If ${\rm Cay}(G,S)$ has FR at some time $t$, then $e^{\imath t}$ is an $n$-th root of unity.
\end{cor}
\begin{proof}
It is a direct consequence of (\ref{f-20}) and Lemma \ref{lem-M}.
\end{proof}

According to (\ref{f-20}), we show the non-existence of FR on integral abelian Cayley graphs if $M=1$.

\begin{prop}Let $\Gamma={\rm Cay}(G,S)$ be an integral abelian Cayley graph. Let $M$ be defined by in (\ref{f-19}). If $M=1$, then $\Gamma$ cannot provide FR between any two vertices.\end{prop}
\begin{proof}If $M=1$, then by (\ref{f-20}), we know that $T\in \mathbb{Z}$ which means that $t=2\pi$. The desired result follows from the fact $H(t)=I_n$.\end{proof}

Thanks to Theorem \ref{lem-2}, if one wants to construct graphs admitting FR between two distinct vertices, the feasible graphs are integral and the difference of two vertices has even order. Using this fact, we propose a general construction of abelian Cayley graphs having FR, as the following corollary reveals. The proof of the following corollary is based on a direct computation. Of course, one can use Theorem \ref{lem-2} to get another proof of this result as well.

%Let $v,\alpha$ and $\beta$ be, respectively, a rational number, a real number and a pure complex number such that $e^{\imath \frac{2\pi}{M} |S|}=e^{\imath 2\pi v}(\alpha+\beta)$.

%If $M>1$ and $\frac{|S|}{M}+v\notin \frac{1}{4}\mathbb{Z}$, then $\Gamma={\rm Cay}(G,S)$ has $e^{\imath 2\pi v}(\alpha,\beta)$-FR between $(x,y)$ and $(x+c,y)$ for every $(x,y)\in G$ and $0\neq c\in \mathbb{Z}_{2^s}$ at the time $t=\frac{2\pi}{M}$.

\begin{cor}\label{thm-1.61} Let $G=\mathbb{Z}_{2}^s\times H$ such that $s\geq 1$ and $H$ is an abelian group of order $\ell$. Let $\Gamma={\rm Cay}(G,S)$ be an integral graph and let $M$ be defined in \eqref{f-19}. If $\frac{|S|}{M}\notin \frac{1}{4}\mathbb{Z}$, then $\Gamma={\rm Cay}(G,S)$ has $\left(\cos\left(\frac{2\pi|S|}{M}\right),\imath\sin\left(\frac{2\pi|S|}{M}\right)\right)$-FR between $(x,y)$ and $(x+a,y)$ for every $(x,y)\in G$ and $0\neq a\in \mathbb{Z}_{2}^s$ at time $t=\frac{2\pi}{M}$.
\end{cor}
\begin{proof}
Denote vertices $(x,y)$ and $(x+a,y)$ by respective $\mu$ and $\nu$. It is well known that $\mathbb{Z}_{2}^s$ is isomorphic to the additive group of $\Ff_{2^s}$, the finite field with $2^s$ elements. Given $g\in\Ff_{2^s}$, an additive character of $\Ff_{2^s}$ is given by $\chi_{g}(h)=(-1)^{\Tr^s_1(gh)}$ with $h\in\Ff_{2^s}$ and $\Tr^s_1(\cdot)$ being the trace function mapping from $\Ff_{2^s}$ to $\Ff_2$. We refer readers to \cite[Chapter 5]{Lidl1997} for more details on additive characters. Below, we view $G$ as $G=\mathbb{F}_{2^s}\times H$.

By \eqref{f-5}, we obtain that
$$
H(t)_{\mu,\mu}=\frac{1}{2^s\ell}\sum_{z\in G} e^{\imath t\lambda_z}\ \ \hbox{and}\ \ H(t)_{\mu,\nu}=\frac{1}{2^s\ell}\sum_{z=(z_1,z_2)\in G} e^{\imath t\lambda_z}(-1)^{\Tr^s_1(az_1)}.
$$
Thanks to $|\{t\in \Ff_{2^s}: \chi_c(t)=0\}|=|\{t\in \Ff_{2^s}: \chi_c(t)=1\}|=2^{s-1}$, we deduce that
\begin{eqnarray*}
\left|H(2\pi/M)_{\mu,\mu}\right|^2+\left|H(2\pi/M)_{\mu,\nu}\right|^2&=&\frac{1}{2^{2s}\ell^2}\sum_{z\in G}\sum_{g\in G} e^{\imath 2\pi \frac{\lambda_z-\lambda_g}{M}}\\
&&+\frac{1}{2^{2s}\ell^2}\sum_{z=(z_1,z_2)\in G}\sum_{g=(g_1,g_2)\in G}e^{\imath 2\pi \frac{\lambda_z-\lambda_g}{M}}(-1)^{\Tr^s_1(az_1+ag_1)}\\
&=&\frac{1}{2^{2s}\ell^2}\sum_{z=(z_1,z_2)\in G}\sum_{g=(g_1,g_2)\in G}\left(1+(-1)^{\Tr^s_1(az_1+ag_1)}\right)\\
&=&1.
\end{eqnarray*}
By \eqref{f-81} and \eqref{f-13}, there exist a rational number $v$, a real number $\alpha$ and a pure complex number $\beta$ such that $e^{\imath \frac{2\pi}{M} |S|}=e^{\imath 2\pi v}(\alpha+\beta)$. It can be easily seen that $\alpha$ and $\beta$ are nonzero if $\frac{|S|}{M}+v\notin \frac{1}{4}\mathbb{Z}$. Putting $v=0$, it follows from $\frac{|S|}{M}\notin \frac{1}{4}\mathbb{Z}$ that $\Gamma={\rm Cay}(G,S)$ has $\left(\cos\left(\frac{2\pi|S|}{M}\right),\imath\sin\left(\frac{2\pi|S|}{M}\right)\right)$-FR at time $t=\frac{2\pi}{M}$.
\end{proof}

\begin{remark}
(1)By the proof of the above corollary, if there exist a rational number $v$, a real number $\alpha$ and a pure complex number $\beta$ such that $e^{\imath \frac{2\pi}{M} |S|}=e^{\imath 2\pi v}(\alpha+\beta)$ and $\frac{|S|}{M}+v\notin \frac{1}{4}\mathbb{Z}$, then the graph $\Gamma={\rm Cay}(G,S)$ defined in Corollary \ref{thm-1.61} has $e^{\imath 2\pi v}(\alpha,\beta)$-FR between $(x,y)$ and $(x+a,y)$ for every $(x,y)\in G$ and $0\neq a\in \mathbb{Z}_{2}^s$ at time $t=\frac{2\pi}{M}$. For simplicity, we state the case $v=0$ in Corollary \ref{thm-1.61}.

(2)Since we only discuss the case $\alpha\beta\neq 0$, the condition $\frac{|S|}{M}\notin \frac{1}{4}\mathbb{Z}$ (resp. $\frac{|S|}{M}+v\notin \frac{1}{4}\mathbb{Z}$) is needed in Corollart \ref{thm-1.61}. Removing the condition $\frac{|S|}{M}\notin \frac{1}{4}\mathbb{Z}$ (resp. $\frac{|S|}{M}+v\notin \frac{1}{4}\mathbb{Z}$), the graph $\Gamma={\rm Cay}(G,S)$ in Corollary \ref{thm-1.61} is said to have PST, FR or be periodic.
\end{remark}

Using Corollary \ref{thm-1.61}, we construct the following family of abelian Cayley graphs having FR.

\begin{prop}\label{thm-1.6}Let $p$ be a prime and let $r\geq 1$ be an integer. Let $G=\mathbb{Z}_{2}\times \mathbb{Z}_{p^r}\times H$ with $H$ being an abelian group of order $m$. Let $S$ be a symmetric subset in $G$ which is defined by
\begin{equation*}
  S=\{(0,\ell,h): 1\leq \ell \leq p^r-1, \gcd(\ell,p)=1,h\in H\}\cup \{(1,0,0)\}.
\end{equation*}
If $p^{r-1}m\neq 1,2,4$, then $\Gamma={\rm Cay}(G,S)$ has FR between $(x,y,z)$ and $(x+1,y,z)$ for every $(x,y,z)\in G$ at the time $t=\frac{2\pi}{p^{r-1}m}$.
\end{prop}
\begin{proof}It is easy to see that the dual group of $G$ is
\begin{equation*}
  \widehat{G}=\{\varphi_u\chi_v\psi_w: u\in \mathbb{Z}_2, v\in \mathbb{Z}_{p^r},w\in H\}
\end{equation*}
where $\varphi_u\chi_v\psi_w(x,y,z)=(-1)^{ux}\omega_{p^r}^{vy}\psi_w(z)$, and $\omega_v=e^{\imath \frac{2\pi}{v}}$ is a primitive $v$-th root of unity. Therefore, the eigenvalues of $\Gamma$ are
\begin{eqnarray*}
  \lambda_{x,y,z}&=&\sum_{(u,v,w)\in S}\varphi_u\chi_v\psi_w(x,y,z)\\
  &=&\sum_{1\leq \ell \leq p^r-1, \gcd(\ell,p)=1}\omega_{p^r}^{\ell y}\sum_{w\in H}\psi_w(z)+(-1)^x, \ \ (x,y,z)\in G.
\end{eqnarray*}
The summation
\begin{equation*}
  \sum_{1\leq \ell \leq p^r-1, \gcd(\ell,p)=1}\omega_{p^r}^{\ell y}:=c(y,p^r)
\end{equation*}
 is in fact a Ramanujan function \cite[Corollary 2.4]{McCarthy1986} and
 \begin{equation*}
   c(y,p^r)=\left\{\begin{array}{ll}
                     -p^{r-1}, & \mbox{ if $v_p(y)={r-1}$}, \\
                     p^{r-1}(p-1) & \mbox{ if $y=0$},\\
                     0, & \mbox{ otherwise}.
                   \end{array}
   \right.
 \end{equation*}
 Here we identify the elements in $\mathbb{Z}_{p^r}$ as integers and $v_p(y)$ is the highest power of $p$ that divides $y$. Then we derive that
 \begin{equation*}
   \lambda_{x,y,z}=\left\{\begin{array}{ll}
                         (-1)^x+(-p^{r-1})m, & \mbox{ if $z=0$ and $v_p(y)={r-1}$}, \\
                         (-1)^x+p^{r-1}(p-1)m, & \mbox{ if $z=0$ and $y=0$}, \\
(-1)^x,& \mbox{ otherwise.}
                        \end{array}
   \right.
 \end{equation*}
Since $a=(x+1,y,z)-(x,y,z)=(1,0,0)$, it is obvious that
 \begin{equation*}
   G_0=\{(0,y,h):y\in \mathbb{Z}_{p^r},h\in H\}, \mbox{ and } G_1=\{(1,y,h):y\in \mathbb{Z}_{p^r},h\in H\}.
 \end{equation*}
We have $M_0=M_1=M=p^{r-1}m$, which are defined by (\ref{f-19}). Due to $\frac{|S|}{M}=p-1+\frac{1}{p^{r-1}m}\notin \frac{1}{4}\mathbb{Z}$ and Corollary \ref{thm-1.61} or Theorem \ref{lem-2} directly, $\Gamma$ has $\left(\cos\left(\frac{2\pi}{p^{r-1}m}\right),\imath\sin\left(\frac{2\pi}{p^{r-1}m}\right)\right)$-FR between $(x,y,z)$ and $(x+1,y,z)$ for every $(x,y,z)\in G$ at the time $t=\frac{2\pi}{p^{r-1}m}$.
\end{proof}

%From the structures of abelian groups, the following corollary is obtained immediately.
%\begin{cor}For every abelian group $G$ of an even order, there is a subset $S$ in $G$ such that the Cayley graph $\Gamma={\rm Cay}(G,S)$ admits FR.\end{cor}
%Below, we present an explicit example to illustrate the requirement that $p^{r-1}m\geq 4$ is necessary in the above Theorem \ref{thm-1.6}.
%\begin{example}Let $G=\mathbb{Z}_2\times \mathbb{Z}_3$ and $S=\{(0,1),(0,2),(1,0)\}$. Then the Cayley graph $\Gamma={\rm Cay}(G,S)$ does not have FR at time $t=2\pi/3$.\end{example}
%\begin{proof}One can check that the spectra of $\Gamma$ are
%\begin{equation*}
%\lambda_{0,0}=3, \lambda_{0,1}=\lambda_{0,2}=0, \lambda_{1,0}=1, \lambda_{1,1}=\lambda_{1,2}=-2.
%\end{equation*}
%Meanwhile,
%\begin{equation*}
%  G_0=\{(0,0),(0,1),(0,2)\}, G_1=\{(1,0),(1,1),(1,2)\}.
%\end{equation*}
%Thus%, by (\ref{f-19}), we have $M_0=M_1=M=3$. Thus
%\begin{equation*}
%  e^{\imath \frac{2\pi}{3}\lambda_g}=\left\{\begin{array}{cc}
%                                           1, & g\in G_0, \\
%                                           \cos(2\pi/3)+\imath \sin(2\pi/3), & g\in G_1.
%                                          \end{array}
%  \right.
%\end{equation*}
%\end{proof}

The following example is used to illustrate Proposition \ref{thm-1.6}.

\begin{example}\label{exam-2} Let $G=\mathbb{Z}_2\times \mathbb{Z}_9$ and $S=\{(0,1),(0,2),(0,4),(0,5),(0,7),(0,8),(1,0)\}$. Then $\Gamma={\rm Cay}(G,S)$ has FR between $(0,0)$ and $(1,0)$ at time $t=2\pi/3$.\end{example}
\begin{proof}The eigenvalues of $\Gamma$ are as follows.
\begin{equation*}
  \lambda_{x,y}=\omega_9^{y}+\omega_9^{2y}+\omega_9^{4y}+\omega_9^{5y}+\omega_9^{7y}+\omega_9^{8y}+(-1)^x, (x,y)\in G.
\end{equation*}
Thus,
\begin{equation*}
  \lambda_{0,0}=7, \lambda_{0,3}=\lambda_{0,6}=-2, \lambda_{0,y}=1, y\in \mathbb{Z}_9^*:=\{1,2,4,5,7,8\}, \mbox{ and}
\end{equation*}
\begin{equation*}
  \lambda_{1,0}=5, \lambda_{1,3}=\lambda_{1,6}=-4, \lambda_{1,y}=-1, y\in \mathbb{Z}_9^*.
\end{equation*}
It is obvious that
\begin{equation*}
  G_0=\{(0,y): y\in \mathbb{Z}_9\},  G_0=\{(1,y): y\in \mathbb{Z}_9\}.
\end{equation*}
Thus, $M_0=M_1=M=3$. Moreover,
\begin{equation*}
  e^{\imath \frac{2\pi}{3}\lambda_g}=\left\{\begin{array}{cc}
                                              e^{\imath \frac{2\pi}{3}}, & g\in G_0, \\
                                              e^{-\imath \frac{2\pi}{3}},  & g\in G_1.
                                            \end{array}
  \right.
\end{equation*}
Thus $\Gamma$ has $(\cos(2\pi/3), \imath \sin(2\pi/3))$-FR at time $t=2\pi/3$.

Indeed, by a Magma program, we find that
\begin{equation*}
  H(2\pi/3)={\rm diag}\left(\underset{9 \ \ cpoies}{\underbrace{{\left(
                                             \begin{array}{cc}
                                               -\frac{1}{2} & \frac{\sqrt{3}}{2} \\
                                               \frac{\sqrt{3}}{2} & -\frac{1}{2} \\
                                             \end{array}
                                           \right), \cdots, \left(
                                             \begin{array}{cc}
                                               -\frac{1}{2} & \frac{\sqrt{3}}{2} \\
                                               \frac{\sqrt{3}}{2} & -\frac{1}{2} \\
                                             \end{array}
                                           \right)}
  }}\right).
\end{equation*}
Thus $\Gamma$ has $(-\frac{1}{2},\frac{\sqrt{3}}{2})$-FR at time $t=2\pi/3$.
\end{proof}

In Proposition \ref{thm-1.6}, if $p^{r-1}m= 1,2\ \mbox{or}\ 4$, the Cayley graph $\Gamma={\rm Cay}(G,S)$ may have FR. We do not have a general result like Corollary \ref{thm-1.6} to guarantee the existence of FR on $\Gamma$ whenever $p^{r-1}m= 1,2\ \mbox{or}\ 4$. Using Theorem \ref{lem-2}, we can check the existence of FR on $\Gamma$ under these cases one by one. The following example shows how to verify it when $p=3$, $r=1$ and $m=1$.

\begin{example}Let $G=\mathbb{Z}_2\times \mathbb{Z}_3$ and $S=\{(0,1),(0,2),(1,0)\}$. Then the Cayley graph $\Gamma={\rm Cay}(G,S)$ has FR at time $t=2\pi/3$.\end{example}
\begin{proof}One can check that the spectra of $\Gamma$ are
\begin{equation*}
  \lambda_{0,0}=3, \lambda_{0,1}=\lambda_{0,2}=0, \lambda_{1,0}=1, \lambda_{1,1}=\lambda_{1,2}=-2.
\end{equation*}
Meanwhile,
\begin{equation*}
  G_0=\{(0,0),(0,1),(0,2)\}, G_1=\{(1,0),(1,1),(1,2)\}.
\end{equation*}
It is easy to check that $M=3$ and
\begin{eqnarray*}
  e^{\imath \frac{2\pi}{3}\lambda_g}&=&\left\{\begin{array}{cc}
                                            1, & g\in G_0, \\
                                           \cos(2\pi/3)+\imath \sin(2\pi/3), & g\in G_1.
                                          \end{array}
  \right.\\
                                     &=&    \left\{\begin{array}{cc}
                                            e^{\imath \frac{\pi}{3}}(\frac{1}{2}-\imath\frac{\sqrt{3}}{2}), & g\in G_0, \\
                                           e^{\imath \frac{\pi}{3}}(\frac{1}{2}+\imath\frac{\sqrt{3}}{2}), & g\in G_1.
                                          \end{array}   \right.
\end{eqnarray*}
By Theorem \ref{lem-2}, $\Gamma$ has $e^{\imath \frac{\pi}{3}}\left(\frac{1}{2},-\imath\frac{\sqrt{3}}{2}\right)$ at time $t=2\pi/3$.
\end{proof}

In Corollary \ref{thm-1.61}, we propose a general construction of Cayley graphs over $G=\mathbb{Z}_{2}^s\times H$ having FR. Next, we discuss the existence of FR on Cayley graphs over $G=\mathbb{Z}_{2^s}\times H$.

\begin{cor}\label{thm-2.81} Let $G=\mathbb{Z}_{2^s}\times H$ such that $s\geq 1$ and $H$ is an abelian group of order $\ell$. Let $\Gamma={\rm Cay}(G,S)$ be an integral graph and let $M$ be defined in \eqref{f-19}. If $\frac{|S|}{M}\notin \frac{1}{4}\mathbb{Z}$, then $\Gamma={\rm Cay}(G,S)$ has $\left(\cos\left(\frac{2\pi|S|}{M}\right),\imath\sin\left(\frac{2\pi|S|}{M}\right)\right)$-FR between $(x,y)$ and $(x+2^{s-1},y)$ for every $(x,y)\in G$ at time $t=\frac{2\pi}{M}$.
\end{cor}
\begin{proof}
The proof of this corollary can be completed by the method analogous to that used in Corollary \ref{thm-1.61}
\end{proof}

At the end of this section, we propose another family of Cayley graphs admitting FR.

\begin{prop}\label{thm-2.8}Let $r_0,r_1,\cdots,r_s$ be positive integers and let $p_0=2, p_1,\cdots,p_s$ be distinct primes such that $\prod_{j=0}^sp_j^{r_j-1}\neq 1,2,4$ with $s\geq 1$. Let $G=\prod_{j=0}^s\mathbb{Z}_{p_j^{r_j}}$ and let $S$ be a subset of $G$ defined by
\begin{equation*}
  S=\{(\ell_0,\ell_1,\cdots, \ell_s): 1\leq \ell_i\leq p_i^{r_i}-1, \gcd(\ell_i,p_i)=1, 0\leq i\leq s\}\cup \{(2^{r_0-1},0,\cdots,0)\}.
\end{equation*}
Let $\Gamma={\rm Cay}(G,S)$ be a Cayley graph. Then $\Gamma$ has FR between $(x_0,x_1,\cdots,x_s)$ and $(x_0+2^{r_0-1},x_1,\cdots,x_s)$ for every $(x_0,x_1,\cdots,x_s)\in G$ at time $2\pi/\prod_{i=0}^sp_i^{r_i-1}$.\end{prop}
\begin{proof}
Due to $a=(x_0+2^{r_0-1},x_1,\cdots,x_s)-(x_0,x_1,\cdots,x_s)=(2^{r_0-1},0,\cdots,0)$, it is easy to see that
\begin{equation*}
  G_0=\{(2x,y_1,\cdots,y_s): x \in \mathbb{Z}_{2^{r_0}}, y_i \in \mathbb{Z}_{p_i^{r_i}}\},  G_1=\{(2x+1,y_1,\cdots,y_s): x \in \mathbb{Z}_{2^{r_0}},  y_i \in \mathbb{Z}_{p_i^{r_i}}\}.
\end{equation*}
For $\mathbf{x}=(x_0,x_1,\cdots,x_s)$, the corresponding eigenvalue of $\Gamma$ is
\begin{equation*}
  \lambda_{\mathbf{x}}=(-1)^{x_0}+\sum_{\underset{1\leq \ell_j \leq p_j^{r_j}-1, \gcd(p_j,\ell_j)=1}{\ell_0,\ell_1,\cdots, \ell_s}}\prod_{j=0}^s\xi_{{p_j}^{r_j}}^{\ell_j x_j}.
\end{equation*}
Thus,
\begin{equation*}
  \lambda_{\mathbf{x}}=\left\{\begin{array}{ll}
                         1+\prod_{i=0}^sp_i^{r_1-1}(p_1-1), & \mbox{ if $(x_0,x_1,\cdots,x_s)=\mathbf{0}$}, \\
                            1+A,& \mbox{ if  $x_j\neq 0,v_{p_j}(x_j)=r_j-1,$ for some $ j\in I\subseteq \{0,\cdots,s\}$}, \\
                        (-1)^{x_0},& \mbox{ if there is some $x_i\neq 0, v_{p_i}(x_i)<{r_i-1}$}.
                        \end{array}
  \right.
\end{equation*}
where $A=(-1)^{|I|}\prod_{j\in I}p_j^{r_j-1}\prod_{i\not\in I}p_i^{r_i-1}(p_i-1)$.
It follows that $M=\prod_{i=0}^sp_i^{r_i-1}$, where $M$ is defined by (\ref{f-19}). Due to $|S|/M\notin \frac{1}{4}\mathbb{Z}$ and Corollary \ref{thm-2.81}, $\Gamma$ has $\left(\cos (2\pi/M), \imath \sin(2\pi/M)\right)$-FR at time $2\pi/M$.
\end{proof}

\begin{remark}(1) It was shown in \cite{Tan_2019} that the order of the underlying group should be doubly-even if an abelian Cayley graph has perfect state transfer. However, for FR, Propositions \ref{thm-1.6} and \ref{thm-2.8} show that there are many more graphs having FR than that for PST on abelian Cayley graphs.

(2) Proposition \ref{thm-2.8} says that for every even number $n$ having a square divisor of an odd prime power, there is an abelian Cayley graph having FR with the underlying group of order $n$. \end{remark}

\section{Plateaued functions over abelian groups and FR on abelian Cayley graphs}\label{S3}

Let $G$ be an abelian group of order $n$ and let $\mathbb{C}^G$ be the set of complex-valued functions on $G$. Then $\mathbb{C}^G$ is an $n$-dimensional vector space. Given $f,g\in \mathbb{C}^G$, define a scalar product as
\begin{equation*}
  \langle f, g\rangle=\sum_{x\in G}f(x)\overline{g(x)}.
\end{equation*}
Then $\mathbb{C}^G$ is a Hilbert space with $\{\frac{1}{\sqrt{n}}\chi_x: x\in G\}$ an orthnormal basis. Thus every function $f$ has the following decomposition.
\begin{equation*}
  f(x)=\sum_{z\in G}\langle f,\frac{1}{\sqrt{n}}\chi_z\rangle \frac{1}{\sqrt{n}}\chi_z(x).
\end{equation*}
For any $f\in \mathbb{C}^G$, its Fourier transform is defined by
\begin{equation*}
  \widehat{f}(\chi_z)=\sum_{x\in G}f(x)\overline{\chi_z(x)}=\langle f, \chi_z\rangle.
\end{equation*}
%The inverse transform is determined by
%\begin{equation*}
%f(x) =\frac{1}{n}\sum_{a\in G}\widehat{f}(\chi_a){\chi_a(x)}=\sum_{a\in G}\langle f, \frac{1}{\sqrt{n}}\chi_a\rangle \frac{1}{\sqrt{n}}\chi_a(x).
%\end{equation*}
\begin{defn}\cite{LOGACHEV1997} Let $G$ be an abelian group of order $n$. A function $f\in \mathbb{C}^G$ is called bent on $G$ if for every $z\in G$, we have
\begin{equation*}
  |\widehat{f}(\chi_z)|=\sqrt{n}.
\end{equation*}
\end{defn}
Bent functions over finite fields of even characteristic were invented and named in 1976 by Rothaus \cite{Rothaus1976}. As a generalization of bent functions, the so-called plateaued functions over finite fields of characteristic two were first introduced by Zheng and Zhang in \cite{Zheng2001}.
\begin{defn}An $n$-variable
Boolean function is said to be $r$-plateaued if the values of its Walsh transform (the Fourier transform of $(-1)^f$)
belong to the set $\{0,\pm 2^{\frac{n+r}{2}}\}$
 for some fixed $r$.\end{defn}

 Bent and plateaued functions have many applications in communications and cryptography, for a comprehensive survey on such functions, interested readers may refer to \cite{Mesnager2016}.

 Let $G$ be an abelian group of order $n$ and let $e=\exp(G)$ be the largest order of the elements in $G$. Then the values of $\chi(z) (z\in G, \chi\in \hat{G})$ are in the cyclotomic field $\mathbb{Q}(\omega_e)$. The Galois group of the extension $\mathbb{Q}(\omega_e)/\mathbb{Q}$  is isomorphic to
\begin{equation*}
  \mathbb{Z}_e^*=\{\ell: 1\leq \ell\leq e-1, \gcd(\ell, e)=1\}.
\end{equation*}
For every $\ell \in \mathbb{Z}_e^*$, the corresponding isomorphic in ${\rm Gal}(\mathbb{Q}(\omega_e)/\mathbb{Q})$ is determined by
\begin{equation*}
  \sigma_\ell: \omega_e\mapsto \omega_e^\ell.
\end{equation*}

\begin{defn}A function $f\in \mathbb{C}^G$ is called a {\it class function} over $G$ if $f$ is integer-valued, namely, $f(x)\in \mathbb{Z}$ for every $x\in G$, and for every $\ell \in  \mathbb{Z}_e^*$, it holds that
\begin{equation}\label{f-161}
   f(\ell x)=f(x), \forall x\in G.
\end{equation}\end{defn}
\begin{lem} \label{lem-3.3}Let $G$ be an abelian group of order $n$ and let $f\in \mathbb{C}^G$ be a class function over $G$. Then for every $\chi\in \hat{G}$, we have
\begin{equation*}
  \hat{f}(\chi)\in \mathbb{Z}.
\end{equation*}
\end{lem}
\begin{proof}Since $f$ is integer-valued,
\begin{equation*}
  \hat{f}(\chi_g)=\sum_{x\in G}f(x)\chi_g(x)
\end{equation*}
is contained in the integer ring $\mathbb{Z}[\omega_e]$, and for every $\ell \in  \mathbb{Z}_e^*$,
\begin{eqnarray*}
  \sigma_\ell(\hat{f}(\chi_g))&=&\sum_{x\in G}f(x)\sigma_\ell(\chi_g(x))=\sum_{x\in G}f(x)(\chi_g(\ell x))\\
  &=&\sum_{x\in G}f(\ell x)(\chi_g(\ell x))\ \ (\mbox{ since $f$ is a class function})\\
  &=&\sum_{x\in G}f(x)\chi_g(x)\\
  &=& \hat{f}(\chi_g).
\end{eqnarray*}
Therefore, $\hat{f}(\chi_g)\in \mathbb{Q}\cap \mathbb{Z}[\omega_e]=\mathbb{Z}$.
\end{proof}
 As an analogue to plateaued functions over finite fields of even characteristic, we define plateaued functions over abelian groups as follows.
 \begin{defn}\label{defn-3.4}Let $G$ be an abelian group of order $n$. Suppose that $p$ is a prime number. Let $f\in \mathbb{C}^G$ be a class function such that for every $\chi\in \widehat{G}$, the Fouerier transform of $f$ at $\chi$ satisfying
  \begin{equation}\label{f-162}
 \widehat{f}(\chi)\equiv k\pmod {p^r}
 \end{equation}
 for some fixed non-negative integers $k,r$. Then $f$ is called $p^r$-plateaued over $G$, where $r$ is the greatest integer $r$ satisfying (\ref{f-162}).\end{defn}
 It is easy to see that the classical plateaued functions over finite fields with even characteristic are special cases of our plateaued functions over abelian groups.
 \begin{example}Let $G=\mathbb{Z}_3\times \mathbb{Z}_3$ and $S$ be a subset of $G$ defined by
  \begin{equation*}
   S=\{(1,0),(2,0),(0,1),(0,2)\}.
\end{equation*}
Let $f$ be the characteristic function of $S$, namely, $f(x)=1$ if $x\in S$ and $0$, otherwise. Then $f$ is a $3$-plateaued function over $G$.
\end{example}
\begin{proof}Indeed, since for every $\ell\in \mathbb{Z}_3^*=\{1,2\}$, we have $\ell S:=\{\ell x: x\in S\}=S$. Thus $f$ is a class function. Moreover,
for every $g=(g_1,g_2)\in G$, we have
\begin{equation*}
  \hat{f}(\chi_g)=\sum_{x\in G}f(x)\chi_g(x)=\omega_3^{g_1}+\omega_3^{-g_1}+\omega_3^{g_2}+\omega_3^{-g_2}.
\end{equation*}
Hence, $\hat{f}(\chi_g)\in \{4,-2,1\}$ and then $\hat{f}(\chi_g)\equiv 1 \pmod 3$. So that $f$ is a $3$-plateaued function over $G$.
\end{proof}

Our main result in this section is as follows.
\begin{thm}\label{thm-3.6} Let $H$ be an abelian group of order $n$ and $S_1$ be a subset of $H$ satisfying $0\notin S_1$ and $\ell S_1=S_1$ for all $\ell \in \mathbb{Z}_e^*$. Let
$G=\mathbb{Z}_2\times H$ and let $S$ be a subset of $G$ defined by
\begin{equation*}
  S=(1,S_1)\cup (0,S_1)\cup \{(1,0)\}.
\end{equation*}
Assume that $p$ divides $d=|S_1|$.
Let $f$ be the characteristic function of $S_1$ and suppose that $f$ is an $p^r$-plateaued function over $H$ with $r\geq 1$. Then the Cayley graph $\Gamma={\rm Cay}(G,S)$ has FR at time $t=\pi/p^{r_0}$ between $(x+1,y)$ and $(x,y)$ for every $(x,y)\in G$, where $r_0=\min(r, v_p(d))\geq 1$.
\end{thm}
\begin{proof}For every $(x,y)\in G$, the corresponding eigenvalue of $\Gamma$ is
\begin{equation*}
  \lambda_{x,y}=(-1)^x+(-1)^x\sum_{z\in S_1}\chi_y(z)+\sum_{z\in S_1}\chi_y(z)=(-1)^x+((-1)^x+1)\hat{f}(\chi_y).
\end{equation*}
Due to $a=(x+1,y)-(x,y)=(1,0)$, it is easy to see that
\begin{equation*}
  G_0=\{(0,h):h\in H\}, \mbox{ and } G_1=\{(1,h):h\in H\}.
\end{equation*}
Thus
\begin{equation*}
  \lambda_{x,y}=\left\{\begin{array}{cc}
                         1+2\hat{f}(\chi_y), & \mbox{ if $(x,y)\in G_0$}, \\
                         -1, & \mbox{ if $(x,y)\in G_1$}.
                       \end{array}
  \right.
\end{equation*}
Therefore, for every $(0,y)\in G_0$, we have
\begin{equation*}
  e^{\imath t \lambda_{0,y}}=e^{\imath \frac{\pi}{p^{r_0}}(\lambda_{0,y}-\lambda_{0,0}+\lambda_{0,0})}=e^{\imath \frac{\pi}{p^{r_0}}(2(\hat{f}(\chi_y)-\hat{f}(\chi_0))+1+2d)}.
\end{equation*}
Since $\hat{f}(\chi_y)-\hat{f}(\chi_0) \equiv 0 \pmod {p^r}$, and $d\equiv 0 \pmod {p^{v_p(d)}}$, we get
\begin{equation*}
  e^{\imath t \lambda_{0,y}}=e^{\imath \frac{\pi}{p^{r_0}}}.
\end{equation*}
Meanwhile, we have
\begin{equation*}
  e^{\imath t \lambda_{1,y}}=e^{-\imath \frac{\pi}{p^{r_0}}}.
\end{equation*}
Thus we get the desired result by Theorem \ref{lem-2}.
\end{proof}
Next we present an example to illustrate Theorem \ref{thm-3.6}.
\begin{example}Let $H=\mathbb{Z}_9$ and let $S_1$ be defined by
\begin{equation*}
  S_1=\{1,2,4,5,7,8\}.
\end{equation*}
Let $G=\mathbb{Z}_2\times \mathbb{Z}_9$ and let $S$ be defined by
\begin{equation*}
  S=\{(0,1),(0,2),(0,4),(0,5),(0,7),(0,8),(1,1),(1,2),(1,4),(1,5),(1,7),(1,8),(1,0)\}.
\end{equation*}
Then $\Gamma={\rm Cay}(G,S)$ has FR at time $t=\pi/3$.
\end{example}
\begin{proof}Let $f$ be the characteristic function of the set $S_1$. Then for every $y\in H$, we have
\begin{equation*}
  \hat{f}(\chi_y)=\sum_{x\in H}f(x)\chi_x(y)=\omega_9^{y}+\omega_9^{2y}+\omega_9^{4y}+\omega_9^{5y}+\omega_9^{7y}+\omega_9^{8y}=c(y,9),
\end{equation*}
where $c(y,9)$ is a Ramanujan function. Thus
\begin{equation*}
  \widehat{f}(\chi_0)=6, \widehat{f}(\chi_1)=\widehat{f}(\chi_2)=\widehat{f}(\chi_4)=\widehat{f}(\chi_5)=\widehat{f}(\chi_7)=\widehat{f}(\chi_8)=0, \widehat{f}(\chi_3)=\widehat{f}(\chi_6)=-3.
\end{equation*}
Therefore, $f$ is a $3$-plateaued function. Moreover $3|d$. By Theorem \ref{thm-3.6}, $\Gamma$ has FR at time $\pi/3$.

We have checked the results by a Magma programme, it shows that
\begin{equation*}
  H(\pi/3)={\rm diag}\left(\underset{9 \ \ copies}{\underbrace{{\left(
                                             \begin{array}{cc}
                                               -\frac{1}{2} & \frac{\sqrt{3}}{2} \\
                                               \frac{\sqrt{3}}{2} & -\frac{1}{2} \\
                                             \end{array}
                                           \right), \cdots, \left(
                                             \begin{array}{cc}
                                               -\frac{1}{2} & \frac{\sqrt{3}}{2} \\
                                               \frac{\sqrt{3}}{2} & -\frac{1}{2} \\
                                             \end{array}
                                           \right)}
  }}\right).
\end{equation*}
This completes the proof.
\end{proof}
Note that the transfer matrix at time $t=\pi/3$ in the above example is exactly the same as that in Example \ref{exam-2}. Combining these two examples, we know that for a given abelian group of even order, there may have different choices of the connection sets such that the resulting graphs have FR at a different time.
% \section{Difference sets and and FR on abelian Cayley graphs}
\section{Cubelike graphs}\label{S4}
Let $\mathbb{F}_{2^n}$ be the finite field with $2^n$ elements. A Cayley graph ${\rm Cay}(G,S)$ is called a cubelike graph if $G$ is the additive group of some finite field $\mathbb{F}_{2^n}$. Since $(\mathbb{F}_{2^n},+)$ is an elementary $2$-group, the classical Hamming graph is a special kind of cubelike graphs.
Note that every nonzero element in $\mathbb{F}_{2^n}$ has order two. Moreover, the map $x\mapsto x+a$ is an isomorphism of the Cayley graph ${\rm Cay}(G,S)$ for every $a\in G$. Thus, in this section, we can fix $a$ as an arbitrary nonzero element in $G$. Thus, in the sequel, we prescribe $a$ as $(1,0,\cdots,0)\in \mathbb{F}_2^n$. Meanwhile, we can assume that $a$ is contained in the connection set.
In this section, we show that many cubelike graphs admit FR. Our first result in this section is as follows.
\begin{thm}\label{thm-4.1}Let $\Gamma={\rm Cay}(\mathbb{F}_{2^n}, S)$ be a cubelike graph. Suppose that
\begin{equation*}
  S=(0,S_0)\cup (1,S_1)\cup \{(1,\mathbf{0})\},
\end{equation*}
where $S_0,S_1\subseteq \mathbb{F}_2^{n-1}$, $\mathbf{0}\not\in S_1$. Denote $|S_0|=d_0$ and $|S_1|=d_1$. If $\min(v_2(d_0+d_1),v_2(d_0-d_1))\geq 3$ and $M\geq 8$, then $\Gamma$ exhibits FR at some time $t$, where $M$ is defined by (\ref{f-19}). \end{thm}
\begin{proof}The additive group of $\mathbb{F}_{2^n}$ is an elementary $2$-group isomorphic to $\mathbb{F}_2^n$, the additive group of the $n$-dimensional vector space over $\mathbb{F}_2$.  Take $\mathbf{a}=(1,\mathbf{0}_{n-1})$, where $\mathbf{0}_{n-1}\in \mathbb{F}_2^{n-1}$. We can assume that $\mathbf{a}\in S$, for otherwise, we use an isomorphism of $\Gamma$ to replace an element in $S$ with $\mathbf{a}$. Then
\begin{equation*}
  G_0=\{(0, \mathbf{x}): \mathbf{x}\in \mathbb{F}_2^{n-1}\}, \mbox{ and }  G_1=\{(1, \mathbf{x}): \mathbf{x}\in \mathbb{F}_2^{n-1}\},
\end{equation*}
where $G_0$ and $G_1$ are defined by (\ref{f-12}). Write $S=(0,S_0)\cup (1,S_1)\cup \{\mathbf{a}\}$ ($0\not\in S_1$) and denote $|S_i|=d_i, i=0,1$. By the assumption that $|S|\equiv 1 \pmod 8$, we can assume that $d_0+d_1+1=1+2^\nu d'$, where $\nu\geq 3$. For every $(0,\mathbf{x})\in G_0$, the corresponding eigenvalue is
\begin{equation*}
  \lambda_{(0,\mathbf{x})}=\sum_{\mathbf{s}_0 \in S_0}(-1)^{\mathbf{s}_0\cdot \mathbf{x}}+\sum_{\mathbf{s}_1 \in S_1}(-1)^{\mathbf{s}_1\cdot \mathbf{x}}+1.%:=\chi_\mathbf{x}(S_0)+\chi_\mathbf{x}(S_1)+1.
\end{equation*}
Particularly, $\lambda_{0,\mathbf{0}}=d_0+d_1+1$.
Meanwhile, for every $(1,\mathbf{x})\in G_1$, the corresponding eigenvalue is
\begin{equation*}
  \lambda_{(1,\mathbf{x})}=\sum_{\mathbf{s}_0 \in S_0}(-1)^{\mathbf{s}_0\cdot \mathbf{x}}-\sum_{\mathbf{s}_1 \in S_1}(-1)^{\mathbf{s}_1\cdot \mathbf{x}}+(-1).%=\chi_\mathbf{x}(S_0)-\chi_\mathbf{x}(S_1)-1.
\end{equation*}
Particularly, $\lambda_{1,\mathbf{0}}=d_0-d_1-1$.

Since $M(\geq 8)$ is a divisor of $ |\mathbb{F}_{2^n}|=2^n$ by Lemma \ref{lem-M}, we can write $M=2^r$, $r\geq 3$.

For a different $\mathbf{x'}\in \mathbb{F}_2^{n-1}$, by the definition of the integer $M$, it is obvious that
\begin{equation*}
   \lambda_{(0,\mathbf{x})}- \lambda_{(0,\mathbf{x'})}\equiv 0 \pmod {2^r}, \lambda_{(1,\mathbf{x})}- \lambda_{(1,\mathbf{x'})}\equiv 0 \pmod {2^r}.
\end{equation*}
%Putting $\mathbf{x}=\mathbf{0}$ and letting $\mathbf{x}'$ run through $\mathbb{Z}_{2}^{n-1}$, we get
%It then follows that
%\begin{eqnarray*}
 % &&d_0+d_1-(\chi_\mathbf{x}(S_0)+\chi_\mathbf{x}(S_1))\equiv 0 \pmod {2^r}, \mbox{ and } \\
%  &&d_0-d_1-(\chi_\mathbf{x}(S_0)-\chi_\mathbf{x}(S_1))\equiv 0 \pmod {2^r}
%\end{eqnarray*}
%Since $d_0+d_1\equiv 0 \pmod 4$, we have $\chi_\mathbf{x}(S_0)+\chi_\mathbf{x}(S_1))\equiv 0\pmod 4$. So that $\chi_\mathbf{x}(S_0)-\chi_\mathbf{x}(S_1))\equiv 0\pmod 2$ which implies that $d_0-d_1\equiv 0 \pmod 2$.
Write $\min(r, v_2(d_0+d_1), v_2(d_0-d_1))=\kappa$. Then $\kappa\geq 3$.
Take $t=\frac{2\pi}{2^{\kappa}}$. Then
\begin{equation*}
  e^{\imath t \lambda_{x,\mathbf{y}}}=\left\{\begin{array}{ll}
                                               e^{\imath \frac{2\pi}{2^\kappa}(d_0+d_1+1)}=e^{\imath \frac{2\pi}{2^\kappa}}, & \mbox{ if $(x,\mathbf{y})\in G_0$}, \\
                                               e^{\imath \frac{2\pi}{2^\kappa}(d_0-d_1-1)}=e^{-\imath \frac{2\pi}{2^\kappa}}, & \mbox{ if $(x,\mathbf{y})\in G_1$}.
                                             \end{array}
  \right.
\end{equation*}
Then we obtain the desired result by Theorem \ref{lem-2}.
\end{proof}
Below we present a simple example as a replenishment to Theorem \ref{thm-4.1}.
\begin{example}\label{exam-4.2} Let $G=\mathbb{F}_2^5$ and let the connection set $S$ be defined by
\begin{equation*}
  S=\{(10000)\}\cup \{(0,s): s\in S_0\}\cup \{(1,s): s\in S_0\}
\end{equation*}
where
\begin{equation*}
  S_0=\{(1100),(0011),(1011),(0111), (1101),(1110)\}.
\end{equation*}
Then $\Gamma={\rm Cay}(G,S)$ has FR at time $t=\pi/4$.
\end{example}
\begin{proof}Take $\mathbf{a}=(10000)$. Then
\begin{equation*}
  G_0=\{(0vxyz): v,x,y,z \in \{0,1\}\},  G_1=\{(1vxyz): v,x,y,z \in \{0,1\}\}.
\end{equation*}
The eigenvalues of $\Gamma$ are
\begin{eqnarray*}
  \lambda_{(uxyz)}&=&(-1)^u+(-1)^{v+x}+(-1)^{y+z}+(-1)^{v+y+z}+(-1)^{x+y+z}+(-1)^{v+x+z}\\
  &&+(-1)^{v+x+y}+(-1)^{u+v+x}+(-1)^{u+y+z}+(-1)^{u+v+y+z}\\
  &&+(-1)^{u+x+y+z}+(-1)^{u+v+x+z}+(-1)^{u+v+x+y}.
\end{eqnarray*}
Thus
\begin{equation*}
  \lambda_{(00000)}=13, \lambda_{(0vxyz)}\in \{5,-3\}, \lambda_{(1vxyz)}=-1, v,x,y,z\in \{0,1\}.
\end{equation*}
It is obvious that
\begin{equation*}
  \lambda_{\mathbf{x}}\equiv \left\{\begin{array}{ll}
                               5 \pmod 8, & \mbox{ if $\mathbf{x}\in G_0$}, \\
                               -1 \pmod 8, & \mbox{ if $\mathbf{x}\in G_1$}.
                             \end{array}\right.
,
\end{equation*}
Take $t=\pi/4$. Then
\begin{equation*}
  e^{\imath t \lambda_\mathbf{x}}=\left\{\begin{array}{ll}
                                           e^{\imath \frac{\pi}{4}}, & \mbox{ if $\mathbf{x}\in G_0$},  \\
                                            e^{-\imath \frac{\pi}{4}}, & \mbox{ if $\mathbf{x}\in G_1$}.
                                         \end{array}
  \right.
\end{equation*}
This completes the proof.
\end{proof}
Our next result is based on bent functions and plateaued functions over $\mathbb{F}_2^n$. Before stating our constructions, we recall some basic properties about such kind of functions.

Let $f: \mathbb{F}_2^n\rightarrow \mathbb{F}_2$ be a nonzero Boolean function. Since the additive group $(\mathbb{F}_2^n,+)$ is isomorphic to the additive group of the finite field $\mathbb{F}_{2^n}$. Then $f$ viewed as a function on $\mathbb{F}_{2^n}$ has a (unique) trace expansion of the form:
\begin{equation*}
  f(x)=\sum_{j\in \Gamma_n}{\rm Tr}^{o(j)}_1(a_jx^j)+\epsilon (1+x^{2^n-1}), \forall x\in \mathbb{F}_{2^n},
\end{equation*}
where $\Gamma_n$ is the set of integers obtained by choosing one element in each cyclotomic coset of 2
modulo $2^n-1$, $o(j)$ is the size of the cyclotomic coset of 2 modulo $2^n-1$ containing $j$, $a_j\in \mathbb{F}_{2^{o(j)}}$
and $\epsilon = wt(f)$ modulo 2 where $wt(f)$ is the Hamming weight of the image vector of $f$, that is, the
number of $x$ such that $f(x) = 1$. Denote ${\rm supp}(f)=\{x\in \mathbb{F}_{2^n}: f(x)=1\}$. It is obvious that the map $f\mapsto {\rm supp}(f)$ gives a one-to-one mapping from the set of Boolean functions to the power set of $\mathbb{F}_{2^n}$.

The Walsh-Hadamard transform of $f$ is defined by
\begin{equation*}
  \widehat{f}_w(a)=\sum_{x\in \mathbb{F}_{2^n}}(-1)^{f(x)+{\rm Tr}^n_1(ax)}, \forall a\in \mathbb{F}_{2^n}.
\end{equation*}
\begin{defn}Let $n=2k$ with $k$ being a positive integer. A Boolean function $f$ is called bent if $\widehat{f}_w(a)\in \{\pm 2^{k}\}$  for all $a\in \mathbb{F}_{2^n}$.\end{defn}
%\begin{defn}Let $m$ be a positive integer. A Boolean function $f$ from $\mathbb{F}_{2^m}$ to $\mathbb{F}_2$ is said to be semi-bent if $\widehat{f}_w(a)\in \{0,\pm 2^{\lfloor \frac{m}{2}\rfloor}\} $. \end{defn}
It is well-known that bent function exists on $\mathbb{F}_{2^{2k}}$ for every $k$. Moreover, for every positive integer $k$, there exist many families of such functions. Semi-bent functions also exist in $\mathbb{F}_{2^n}$ for all integer $m\geq 1$, e.g., in \cite{Mesnager2016}.

Let $f: \mathbb{F}_2^n\rightarrow \mathbb{F}_2$ be a Boolean function. $S={\rm supp}(f)\subseteq \mathbb{F}_{2^n}$. For the cubelike graph $\Gamma={\rm Cay}(\mathbb{F}_{2^n},S)$, its eigenvalues are $\lambda_0=|S|$ and
\begin{equation}\label{f-e14}
 \lambda_x=\sum_{z\in S}\chi_x(z)=\sum_{y\in \mathbb{F}_{2^m}}\frac{1-(-1)^{f(y)}}{2}\chi_x(y)=-\frac{1}{2}\widehat{f}_w(x), 0\neq x\in \mathbb{F}_{2^n}.
\end{equation}

For bent and semi-bent functions, the following results are known in the literature, e.g., on pages 72 and 422 in \cite{Mesnager2016}.

\begin{lem}\label{lem-4.4} Let $n=2k$ be a positive even integer. Let $f$ be a Boolean function on $\mathbb{F}_{2^n}$ and let $S={\rm supp}(f)$.

(1) If $f$ be a bent function, then $|S|=2^{n-1}\pm 2^{k-1}$. Moreover, define a function $\widetilde{f}$, called the dual of $f$, by
\begin{equation*}
  \widehat{f}_w(x)=2^k(-1)^{\widetilde{f}(x)}.
\end{equation*}
Then $\widetilde{f}$ is also a bent function. As a consequence, the numbers of occurrences of $\widehat{f}_w(x)$ taking the values $\pm 2^k$ are $2^{n-1}\pm 2^{k-1}$ or $2^{n-1}\mp 2^{k-1}$.

 (2) If $f$ is a semi-bent function, then $|S|\in \{2^{n-1}, 2^{n-1}\pm 2^{k-1}\}$. Moreover, we have the following table.
 \begin{center}Table 1. Walsh spectrum
of semi-bent functions
$ f$ with $f(0)=0$
\begin{tabular}{|c|c|}
  \hline
  Value of $\widehat{f}(x),x\in \mathbb{F}_{2^m}$ & Frequency \\
    \hline
 $0$ & $2^{n-1}+2^{n-2}$ \\
   \hline
  $2^{k+1}$ & $2^{n-3}+2^{k-2}$ \\
    \hline
   $-2^{k+1}$ & $2^{n-3}-2^{k-2}$ \\
  \hline
\end{tabular}
\end{center} \end{lem}

It is known also that if $f$ is a bent function, then so is its complementary function, i.e, $g=1+f$.

The following two constructions are actually corollaries of Theorem \ref{thm-3.6}.
\begin{cor}\label{bentconstruction}Let $n=2k\geq 4$ for a positive integer $k$ and $f$ be a bent (resp. semi-bent) function on $\mathbb{F}_2^n$. Denote $S_0=supp(f)$ and define
\begin{equation*}
  S=\{(1\mathbf{0}_{n})\}\cup \{(0,s):\mathbf{s}\in S_0\}\cup \{(1,s): \mathbf{s}\in S_0\}\subset \mathbb{F}_2^{n+1}.
\end{equation*}
Then the cubelike graph $\Gamma={\rm Cay}(\mathbb{F}_2^{n+1},S)$ has FR at the time $t=\frac{\pi}{2^k}$ (resp. $t=\frac{\pi}{2^{k+1}}$ ).
\end{cor}
\begin{proof}We just give the proof for the case that $f$ is a bent function. The case for the semi-bent functions can be proved similarly.
Obviously,
\begin{equation*}
  G_0=\{(0\mathbf{y}): \mathbf{y}\in \mathbb{F}_2^n\},  G_1=\{(1\mathbf{y}): \mathbf{y}\in \mathbb{F}_2^n\}.
\end{equation*}
It is easily seen that the eigenvalues of $\Gamma$ are
\begin{equation*}
  \lambda_{(x\mathbf{y})}=(-1)^x+((-1)^x+1)\sum_{\mathbf{s}\in S_0}(-1)^{\mathbf{s}\cdot \mathbf{y}}, x\in \mathbb{F}_2, \mathbf{y}\in \mathbb{F}_2^n.
\end{equation*}
Thus, for every $(1\mathbf{y})\in G_1$, $\lambda_{1\mathbf{y}}=-1$. Moreover, for $\mathbf{y}\neq \mathbf{0}$, $(0\mathbf{y})\in G_0$, the corresponding eigenvalue is
\begin{equation*}
  \lambda_{(0\mathbf{y})}=1-\hat{f}_w(\mathbf{y})=1\pm 2^k.
\end{equation*}
When $\mathbf{y}=\mathbf{0}_n$, the corresponding trivial eigenvalue is $\lambda_{(0\mathbf{0}_n)}=1+(2^n\pm 2^k)$.
One can check that $\lambda_{(0\mathbf{y})} \equiv 1+2^k \pmod {2^{k+1}}$. Take $t=\frac{\pi}{2^k}$. We have
\begin{equation*}
  e^{\imath t \lambda_\mathbf{x}}=\left\{\begin{array}{cc}
                                           e^{\imath \frac{\pi}{2^k}}, & \mbox{ if $\mathbf{x}\in G_0$}, \\
                                          e^{-\imath \frac{\pi}{2^k}}, & \mbox{ if $\mathbf{x}\in G_1$}.
                                         \end{array}
  \right.
\end{equation*}
This completes the proof.
\end{proof}
Note that when $k=2$ in the above Corollary \ref{bentconstruction}, we can choose a bent function $f(x_1x_2x_3x_4)=x_1x_2+x_3x_4$, then we have $supp(f)=\{(1100),(0011),(1011),(0111), (1101),(1110)\}$. In this case, we get the graph in Example \ref{exam-4.2}.

We can also use plateaued functions to construct cubelike graphs that admit FR. The main idea is that presented in Theorem \ref{thm-3.6}. The detail of which is omitted.

\section{Concluding remarks}\label{S5}

Using the spectral decomposition of adjacency matrices of abelian Cayley graphs, we characterized abelian Cayley graphs having FR. This characterization indicates potential abelian Cayley graphs admitting FR. We established a general construction of integral graphs ${\rm Cay}(\mathbb{Z}_{2^s}\times H,S)$ having FR. Meanwhile, we constructed several new families of abelian Cayley graphs having FR. The notation of plateaued functions over finite fields was generalized to finite abelian groups. Based on plateaued functions, a new family of abelian Cayley graphs having FR was built. This paper and reference \cite{Tan_2019} show that bent, semi-bent and plateaued functions are powerful ingredients in the constructions of abelian Cayley graphs having PST or FR. It is important to study properties of plateaued functions over finite abelian groups. Another particularly interesting direction is to construct Cayley graphs over non-abelian groups having FR.

\section*{Acknowledgement}

%We gratefully acknowledge the referee for their
%valuable comments and suggestions which improve the quality of the paper. Special thanks goes to the reviewer for leading us to know the references \cite{chan, cou, cg}.

\bibliographystyle{elsarticle-num}

\bibliography{RefPST}

\end{document}